\documentclass{amsart}

\usepackage{graphicx}
\usepackage{amssymb}
\usepackage{float}

\newtheorem{theorem}{Theorem}[section]
\newtheorem{lemma}[theorem]{Lemma}

\newtheorem{proposition}[theorem]{Proposition}

\theoremstyle{definition}

\theoremstyle{remark}
\newtheorem{remark}[theorem]{Remark}

\numberwithin{equation}{section}

\newcommand{\R}{{\mathbb R}}

\newcommand{\li}{\textnormal{li}}

\begin{document}

\title[Numbers free of large primes, explicitly]{Explicit estimates for the distribution of numbers free of large prime factors}

\author{Jared D. Lichtman}
\address{Department of Mathematics, Dartmouth College, Hanover, NH 03755}

\email{lichtman.18@dartmouth.edu}

\author{Carl Pomerance}
\address{Department of Mathematics, Dartmouth College, Hanover, NH 03755}
\email{carl.pomerance@dartmouth.edu}

%    General info
%\subjclass[2000]{Primary 11Y11; Secondary 11A51}

\date{\today}

%\dedicatory{This paper is dedicated to}

%\keywords{Fermat test, Miller--Rabin test, probable prime}

\begin{abstract}
There is a large literature on the asymptotic distribution of
numbers free of large prime factors, so-called {\it smooth} or
{\it friable} numbers.  But there is very little known about
this distribution that is numerically explicit.  In this paper
we follow the general plan for the saddle point argument of
Hildebrand and Tenenbaum, giving explicit and fairly
tight intervals in which the true count lies.  We give two
numerical examples of our method, and with the larger one, our interval is
so tight we can exclude the famous Dickman--de Bruijn asymptotic estimate
as too small and the Hildebrand--Tenenbaum main term as too large.
\end{abstract}

\maketitle

%% The correct journal style for \specialsection is all uppercase; a known bug
%% in amsart.cls prevents this, so input must be uppercase until it is fixed.

\section{Introduction}
For a positive integer $n>1$, denote by $P(n)$ the largest prime factor of $n$,
and let $P(1)=1$.
Let $\Psi(x,y)$ denote the number of $n\le x$ with $P(n)\le y$.
Such integers $n$ are known as $y$-smooth, or $y$-friable. 
Asymptotic estimates for $\Psi(x,y)$ are quite useful in many applications,
not least of which is in the analysis of factorization and discrete
logarithm algorithms.

One of the earliest results is due to Dickman \cite{D} in 1930, who gave 
an asympotic formula for $\Psi(x,y)$ in the case that $x$ is a fixed power of $y$. 
Dickman showed that
\begin{align}
\label{eq:dd}
\Psi(x,y) \sim x\rho(u)\qquad(y\to\infty,~x=y^u)
\end{align}
for every fixed $u\ge1$, where $\rho(u)$ is the ``Dickman--de Bruijn" function, defined to be the continuous solution of the delay differential equation
\begin{align*}
u\rho'(u) + \rho(u-1) & = 0\qquad(u>1),\\
\rho(u) & = 1\qquad(0\le u\le 1).
\end{align*}

There remain the questions of the error in the approximation \eqref{eq:dd},
and also the case when $u=\log x/\log y$ is allowed to grow with $x$ and $y$. 
In 1951, de Bruijn \cite{dB2} proved that
\begin{align*}
\Psi(x,y) = x\rho(u)\Big(1 + O_\varepsilon\Big(\frac{\log(1+u)}{\log y}\Big)\Big)
\end{align*}
holds uniformly for $x\ge 2$, $\exp\{(\log x)^{5/8+\varepsilon}\} < y\le x$,
 for any fixed $\varepsilon>0$. 
After improvements in the range of this result by Maier and Hensley,
Hildebrand \cite{H} showed that the de Bruijn estimate holds when
$\exp(\{(\log\log x)^{5/3+\varepsilon}\})\le y\le x$.

%The behavior of $\Psi(x,y)$ is quite different when $y$ is very small compared to $x$. Ennola showed that the estimate
%\begin{align}
%\Psi(x,y) = \frac{1}{\pi(y)!}\prod_{p\le y}\Big(\frac{\log x}{\log p}\Big)\Big(1+O\Big(\frac{y^2}{\log x\log y}\Big)\Big)
%\end{align}
%holds uniformly for
%\begin{align}
%2\le y\le \sqrt{\log x}.
%\end{align}

%Up to this point, $\Psi(x,y)$ was well understood in the extreme cases when $u$ was very large and very small. In 1986, Hildebrand and Tenenbaum closed this ``gap" by providing a uniform estimate for $\Psi(x,y)$ for all $x\ge y\ge 2$, yielding an asymptotic formula when $y$ and $u$ tend to infinity.

In 1986, Hildebrand and Tenenbaum \cite{HT1} provided a uniform estimate for 
$\Psi(x,y)$ for all $x\ge y\ge 2$, yielding an asymptotic formula when $y$ and 
$u$ tend to infinity.  The starting point for their method is an elementary
argument of Rankin \cite{R} from 1938, commonly known now as Rankin's ``trick".
For complex $s$,
define 
$$\zeta(s,y) = \sum_{\substack{n\ge1\\P(n)\le y}}n^{-s} = \prod_{p\le y}(1-p^{-s})^{-1}$$
(where $p$ runs over primes) as the partial Euler product of the Riemann zeta function $\zeta(s)$.  Then
for $0<\sigma<1$, we have
\begin{equation}
\label{eq:R}
\Psi(x,y)=\sum_{\substack{n\le x\\P(n)\le y}}1
\le\sum_{P(n)\le y}(x/n)^\sigma=x^\sigma\zeta(\sigma,y).
\end{equation}
Then $\sigma$ can be chosen optimally to minimize $x^\sigma\zeta(\sigma,y)$.

Let
$$\phi_j(s,y) = \frac{\partial^j}{\partial s^j}\log\zeta(s,y).$$
The function 
$$
\phi_1(s,y)=-\sum_{p\le y}\frac{\log p}{p^s-1}
$$
is especially useful since the solution $\alpha=\alpha(x,y)$ to $\phi_1(\alpha,y)+\log x=0$
gives the optimal $\sigma$ in \eqref{eq:R}.
We also denote $\sigma_j(x,y) = |\phi_j(\alpha(x,y),y)|$.

In this language, Hildebrand and Tenenbaum \cite{HT1} proved that the estimate
\begin{align*}
%\label{eq:HT}
\Psi(x,y) = \frac{x^\alpha\zeta(\alpha,y)}{\alpha\sqrt{2\pi\sigma_2(x,y)}}\Big(1 + O\Big(\frac{1}{u} + \frac{\log y}{y}\Big)\Big)
\end{align*}
holds uniformly for $x\ge y\ge 2$. As suggested by this formula, quantities $\alpha(x,y)$ and $\sigma_2(x,y)$ are of interest in their own right, and were given uniform estimates which imply the formulae
\begin{align*}
\alpha(x,y) \sim \frac{\log(1+y/\log x)}{\log y}
\end{align*}
and
\begin{align*}
\sigma_2(x,y) \sim \Big(1+\frac{\log x}{y}\Big)\log x\log y,
\end{align*}
together which imply
\begin{align*}
\Psi(x,y) & \sim \frac{x^\alpha\zeta(\alpha,y)}{\sqrt{2\pi u}\log(y/\log x)}\qquad (\text{if } y/\log x\to \infty),\\
\Psi(x,y) & \sim \frac{x^\alpha\zeta(\alpha,y)}{\sqrt{2\pi y/\log y}}\qquad (\text{if } y/\log x\to 0).
\end{align*}

These formulae indicate that $\Psi(x,y)$ undergoes a ``phase change" when $y$ is of order $\log x$, 
see \cite{dB1}.
This paper concentrates on the range where $y$ is considerably larger, say $y > (\log x)^4$.

 The primary aim of this paper is to make the Hildebrand--Tenenbaum
 method explicit and 
so effectively construct an algorithm for obtaining good bounds for 
$\Psi(x,y)$.

\subsection{Explicit Results}
Beyond the Rankin upper bound
$\Psi(x,y)\le x^\alpha\zeta(\alpha,y)$, we have the explicit lower bound
\begin{align*}
\Psi(x,y) \ge x^{1-\log\log x/\log y}=\frac{x}{(\log x)^u}
\end{align*}
due to Konyagin and Pomerance \cite{KP}.
Recently Granville and Soundararajan \cite{GS} found an elementary improvement
of Rankin's upper bound, which they have graciously permitted us to include
in an appendix in this paper.  In particular, they show that
$$
\Psi(x,y)\le 1.39y^{1-\sigma}x^\sigma\zeta(\sigma,y)/\log x
$$
for every value of $\sigma\in[1/\log y,1]$, see Theorem \ref{thm:GS}.

In another direction, by relinquishing the goal of a compact formula, 
several authors have devised algorithms to compute bounds on $\Psi(x,y)$ for given $x,y$ as inputs. For example, using an accuracy parameter $c$, Bernstein \cite{Bern} created an algorithm to generate bounds $B^-(x,y)\le \Psi(x,y) \le B^+(x,y)$ with
\begin{align*}
\frac{B^-}{\Psi} \ge 1 - \frac{\log x}{c\log3/\log2} \quad\text{and}\quad \frac{B^+}{\Psi} \le 1 + \frac{2\log x}{c\log3/\log2},
\end{align*}
running in 
\begin{align*}
O\Big(\frac{y}{\log_2 y} + \frac{y\log x}{\log^2 y} + c\log x\log c\Big)
\end{align*}
time. Parsell and Sorenson \cite{PS} refined this algorithm to run in
\begin{align*}
O\Big(c\frac{y^{2/3}}{\log y} + c\log x\log c\Big)
\end{align*}
time, as well as obtaining faster and tighter bounds assuming the Riemann Hypothesis.  The largest example computed by this method was an approximation of
$\Psi(2^{255},2^{28})$.

As seen in 
Figure \ref{fig:ex}, the bounds presented in this paper far outshine best-known 
upper and lower bounds for the two examples presented. 
We also provide the main term estimates $x^\alpha\zeta(\alpha,y)/\alpha\sqrt{2\pi\sigma_2}$ 
from \cite{HT1} and $\rho(u)x$ from \cite{D} as points of reference. 
It is interesting that our estimates in the second example are closer to the
truth than are the Dickman--de Bruijn and Hildebrand--Tenenbaum main terms.
The second-named author has asked if 
$\Psi(x,y)\ge x\rho(u)$ holds in general for $x\ge2y\ge2$, see \cite[(1.25)]{G}.
This inequality is known for $u$ bounded and $x$ sufficiently large,
see the discussion in \cite[Section 9]{M}.

\begin{figure}[H]
  \caption{Examples.} 
  \[\begin{array}{c|cc} \label{fig:ex}
x &  10^{100} &  10^{500} \\
y &  10^{15} &  10^{35} \\
\hline\\
\text{KP}  &  1.786\cdot10^{84} &  1.857\cdot10^{456}\\
\text {R}  &  4.599\cdot10^{96} &  9.639\cdot10^{484}\\
\text{GS} & 5.350\cdot10^{95} & 6.596\cdot10^{483}\\
\text{DD} & 2.523\cdot10^{94} &  1.472\cdot10^{482} \\
\text{HT}& 2.652\cdot10^{94} & 1.5127\cdot10^{482} \\
\Psi^- &  2.330\cdot10^{94}&  1.4989\cdot10^{482} \\
\Psi^+ & 2.923\cdot10^{94} &  1.5118\cdot10^{482} \\
\end{array}\]
\end{figure}

Here, 
\begin{align*}
&\hbox{ KP is the Konyagin--Pomerance lower bound }x/(\log x)^u,\\
&\hbox{ R is the Rankin upper bound }x^\alpha\zeta(\alpha,y),\\
&\hbox{ GS is the Granville--Soundararajan upper bound
}1.39y^{1-\alpha}x^\alpha\zeta(\alpha,y)/\log x,\\
&\hbox{ DD is the Dickman--de Bruijn main term }\rho(u)x, \hbox{ and}\\
&\hbox{ HT is the
Hildebrand--Tenenbaum main term }x^\alpha\zeta(\alpha,y)/(\alpha\sqrt{2\pi\sigma_2}). 
\end{align*}

Our principal result, which benefits from some notation developed over the
course of the paper, is Theorem \ref{thm:main}.  It is via this theorem
that we were able to estimate $\Psi(10^{100},10^{15})$ and $\Psi(10^{500},10^{35})$ as in the table above.

\section{Plan for the paper}

The basic strategy of the saddle-point method relies on Perron's formula, 
which implies the identity
\begin{align*}
    \Psi(x,y) = \frac{1}{2\pi i}\int_{\sigma-i\infty}^{\sigma+i\infty} \zeta(s,y)\frac{x^s}{s}\;ds,
\end{align*}
for any $\sigma>0$.   It turns out that the best value of $\sigma$ to
use is $\alpha=\alpha(x,y)$ discussed in the Introduction.
We are interested in abridging the 
integral at a certain height $T$ and then approximating the contribution 
given by the tail. 
%Hildebrand and Tennenbaum \cite{HT1} were able to obtain rapid convergence of 
%the tail by setting $\sigma=\alpha(x,y)$, the saddle-point of 
%$x^\sigma\zeta(\sigma,y)$. 
%
%As with \cite{HT1}, we begin by truncating the infinite Perron integral
%at a convenient point $T$:
To this end, we have
\begin{equation}\label{eq:Per}
 \Psi(x,y) = \frac{1}{2\pi i}\int_{\alpha-iT}^{\alpha+iT} \zeta(s,y)\frac{x^s}{s}\;ds + \text{Error}.
\end{equation}
There is a change in behavior occurring in $\zeta(s,y)$ when $t$ is on the order $1/\log y$. In \cite{HT1} it is shown that
\begin{align}
\Big|\frac{\zeta(s,y)}{\zeta(\alpha,y)}\Big| & = \prod_{p\le y}\Big|\frac{1-p^{-\alpha}}{1-p^{-s}}\Big| = \prod_{p\le y}\Big(1+\frac{2(1-\cos(t\log p))}{p^\alpha(1-p^{-\alpha})^2} \Big)^{-1/2}\notag\\
& \le \exp\Big\{-\sum_{p\le y}\frac{1-\cos(t\log p)}{p^\alpha} \Big\}.
\label{eq:zetaineq}
\end{align}
Thus when $t$ is small (compared to $1/\log y$) the oscillatory terms are in 
resonance, and when $t$ is large the oscillatory terms should exhibit 
cancellation. This behavior suggests we should divide our range of 
integration into $|t|\le T_0$ and $T_0<|t|< T$, where $T_0\approx1/\log y$ is 
a parameter to be optimized.  

The contribution for $|t|\le T_0$ will
constitute a ``main terrm", and so we will try to estimate this part
very carefully.
In this range we forgo \eqref{eq:zetaineq} and
attack the integrand $\zeta(s,y)x^s/s$ 
directly. The basic idea is to expand $\phi(s,y)=\log\zeta(s,y)$ as a 
Taylor series in $t$. This approach, when carefully done, gives us fairly
close upper and lower bounds for the integral.  In our smaller example, 
the upper bound is less than 1\% higher than the lower bound,
and in the larger example, this is better by a factor of 20.  Considerably more
noise is encountered beyond $T_0$ and in the Error in \eqref{eq:Per}.

For the second range $T_{0}<|t|< T$, we focus on obtaining a satisfactory lower bound on the sum over primes,
\begin{align*}
    \sum_{p\le y}\frac{1-\cos(t\log p)}{p^\alpha}.
\end{align*}
Our strategy is to sum the first $L$ terms directly, and then obtain an analytic formula $W(y,w)$ to lower bound the remaining terms starting at some $w\ge L$, where essentially
\begin{align*}
    W(y,w) = \frac{y^{1-\alpha} - w^{1-\alpha}}{1-\alpha} + \text{error}.
\end{align*}

With an explicit version of Perron's formula, the Error in \eqref{eq:Per}
may be handled by
\begin{align*}
    \big|\text{Error}\big| & \le x^\alpha\sum_{P(n)\le y}\frac{1}{n^\alpha}\min\Big(1,\frac{1}{\pi T|\log(x/n)|}\Big)\\
    & \le x^\alpha \sum_{\substack{P(n)\le y\\T|\log(x/n)|>T^d}}\frac{1}{n^\alpha}\frac{1}{\pi T|\log(x/n)|} + \sum_{\substack{P(n)\le y\\T|\log(x/n)|\le T^d}}\Big(\frac{x}{n}\Big)^\alpha\\
    & \le \frac{x^\alpha\zeta(\alpha,y)}{\pi T^d} + e^{\alpha T^{d-1}}\Big[\Psi(xe^{T^{1-d}},y) - \Psi(xe^{-T^{1-d}},y)\Big].
\end{align*}
Here $d\approx\frac12$ is a parameter of our choosing, which 
we set to balance the two terms above. Thus the problem of bounding 
$|$Error$|$ is reduced to estimating the number of $y$-smooth integers in 
the ``short" interval $\big(xe^{-T^{1-d}}, xe^{T^{1-d}}\big]$.

This latter portion is better handled when $T$ is large, but the earlier
portion in the range $[T_0,T]$ is better handled when $T$ is small.
Thus, $T$ is numerically set to balance these two forces.

In our proofs we take full advantage of some recent calculations involving the prime-counting function $\pi(x)$
and the Chebyshev functions 
$$
\psi(x)=\sum_{p^m\le x}\log p,\quad \vartheta(x)=\sum_{p\le x}\log p,
$$
with $p$ running over primes and $m$ running over positive integers.
  As a corollary of the papers \cite{B1}, \cite{B2} of
B\"uthe we have the following excellent result.
\begin{proposition}\label{prop:epnt}
For $1427\le x\le 10^{19}$ we have
$$
.05\sqrt{x}\le x-\vartheta(x)\le 1.95\sqrt{x}.
$$
We have
\begin{align*}
%\label{eq:largex}
\frac{|\vartheta(x)-x|}{x} <
\begin{cases}2.3\cdot10^{-8}, &\hbox{ when }x>10^{19},\\
1.2\cdot 10^{-8}, &\hbox{ when }x> e^{45},\\
1.2\cdot 10^{-9}, &\hbox{ when }x> e^{50},\\
2.9\cdot 10^{-10}, &\hbox{ when }x> e^{55}.
\end{cases}
\end{align*}
\end{proposition}
\begin{proof}
The first assertion is one of the main results in B\"uthe \cite{B2}.  Let $H$ be a number such that
all zeros of the Riemann zeta-function with imaginary parts in $[0,H]$ lie on the $1/2$-line.
Inequality (7.4) in B\"uthe \cite{B1} asserts that if $x/\log x\le H^2/4.92^2$ and $x\ge5000$, then
$$
\frac{|\vartheta(x)-x|}{x} < \frac{(\log x-2)\log x}{8\pi\sqrt{x}}.
$$
We can take $H=3\cdot 10^{10}$, see Platt \cite{P}.  Thus, we have the result in the range
$10^{19}\le x\le e^{45}$.  For $x\ge e^{45}$ we have from B\"uthe \cite{B1} that
$|\psi(x)-x|/x\le 1.118\cdot 10^{-8}$.  Further, we have (see \cite[(3.39)]{RS1}) for $x>0$,
$$
\psi(x)\ge\vartheta(x)>\psi(x)-1.02x^{1/2}-3x^{1/3}.
$$
(This result can be improved, but it is not important to us.)
Thus, for $x\ge e^{45}$ we have $|\vartheta(x)-x|/x \le 1.151\cdot 10^{-8}$, establishing our
result in this range.  For the latter two ranges we argue similarly, using $|\psi(x)-x|\le 1.165\cdot 10^{-9}$
when $x\ge e^{50}$ and $|\psi(x)-x|\le 2.885\cdot 10^{-10}$ for $x\ge e^{55}$, both of these inequalities
coming from \cite{B1}.
\end{proof}

We remark that there are improved inequalities at higher values of $x$, found in \cite{B1}
and \cite{FK}, which one would want to use if estimating $\Psi(x,y)$ for larger values of $y$
than we have done here.

\section{The main argument}

As in the Introduction, for complex $s$, define 
$$\zeta(s,y) = \sum_{\substack{n\ge1\\P(n)\le y}}n^{-s} = \prod_{p\le y}(1-p^{-s})^{-1},$$
which is the Riemann zeta function restricted to $y$-smooth numbers, and for $j\ge0$, let
$$\phi_j(s,y) = \frac{\partial^j}{\partial s^j}\log\zeta(s,y).$$
We have the explicit formulae,
\begin{align*}
    \phi_1(s,y) & = -\sum_{p\le y}\frac{\log p}{p^s-1},\\
    \phi_2(s,y) & = \sum_{p\le y}\frac{p^s\log^2 p}{(p^s-1)^2},\\
    \phi_3(s,y) & = -\sum_{p\le y}\frac{(p^{2s}+p^s)\log^3 p}{(p^s-1)^3},\\
    \phi_4(s,y) & = \sum_{p\le y}\frac{(p^{3s}+4p^{2s}+p^s)\log^4 p}{(p^s-1)^4},\\
    \phi_5(s,y) & = -\sum_{p\le y}\frac{(p^{4s}+11p^{3s}+11p^{2s}+p^s)\log^5 p}{(p^s-1)^5}.
\end{align*}
Note that for $y\ge2$, $\sigma>0$, $\phi_1(\sigma,y)$ is strictly increasing from $0$, so
there is a unique solution $\alpha=\alpha(x,y)>0$ to the equation 
$$\log x + \phi_1(\alpha,y) = 0.$$
Since we cannot exactly solve this equation, we shall assume any choice of $\alpha$ that we use
is a reasonable approximation to the exact solution, and we must take into account an upper bound
for the difference between our value and the exact value.
We denote 
$$
\phi_j=\phi_j(\alpha,y),\quad \sigma_j=|\phi_j|=(-1)^j\phi_j,\quad
%\sigma_j=\sigma_j(\alpha,y) = (-1)^j\phi_j(\alpha,y)=|\phi_j(\alpha,y)|,\quad
B_j=B_j(t)=\sigma_jt^j/j!
$$
so that the Taylor series of $\phi(s,y) = \log\zeta(s,y)$ about $s=\alpha$ is
$$\phi(\alpha+it,y) = \sum_{j\ge0}\frac{\sigma_j}{j!}(-it)^j = \sum_{j\ge0}(-i)^jB_j.$$

Our first result, which is analogous to Lemma 10 in \cite{HT1}, sets the stage for our estimates.

\begin{lemma} \label{lm:10}
Let 
%$.7<\alpha<1$, 
$0<d<1$ and $T>1$. We have that
\begin{align*}
\Big| \Psi(x,y) & - \frac{1}{2\pi i}\int_{\alpha-iT}^{\alpha+iT} \zeta(s,y)\frac{x^s}{s}\;ds\Big| \\
& \le \frac{x^\alpha\zeta(\alpha,y)}{\pi T^d} + e^{\alpha T^{d-1}}\Big[\Psi(xe^{T^{d-1}},y) - \Psi(xe^{-T^{d-1}},y)\Big].
\end{align*}
\end{lemma}
\begin{proof}
We have
\begin{align*}
    \frac{1}{2\pi i}\int_{\alpha-iT}^{\alpha+iT} \zeta(s,y)\frac{x^s}{s}\;ds &
= \frac{1}{2\pi i}\int_{\alpha-iT}^{\alpha+iT} \sum_{P(n)\le y} \Big(\frac{x}{n}\Big)^s\frac{ds}{s}\\
&= \sum_{P(n)\le y}\frac{1}{2\pi i}\int_{\alpha-iT}^{\alpha+iT} \Big(\frac{x}{n}\Big)^s\frac{ds}{s},
\end{align*}
where the interchange of sum and integral is justified since $\zeta(s,y)$ is a finite product, hence uniformly convergent as a sum. 

By Perron's formula (see \cite[\S 11.12]{A}), we have
%equations (9) and (10) in Ingham (Perron's Formula), we have
\begin{align*}
    \Big|\frac{1}{2\pi i}\int_{\alpha-iT}^{\alpha+iT} \Big(\frac{x}{n}\Big)^s\frac{ds}{s}\Big| \le \frac{(x/n)^\alpha}{\min\Big(1,\pi T|\log(x/n)|\Big)}\qquad\text{ if }n > x,\\
    \Big|1-\frac{1}{2\pi i}\int_{\alpha-iT}^{\alpha+iT} \Big(\frac{x}{n}\Big)^s\frac{ds}{s}\Big| \le \frac{(x/n)^\alpha}{\min\Big(1,\pi T|\log(x/n)|\Big)}\qquad\text{ if }n\le x.
\end{align*}
Together these imply
\begin{align*}
    \Big|\Psi(x,y) -& \frac{1}{2\pi i}\int_{\alpha-iT}^{\alpha+iT} \zeta(s,y)\frac{x^s}{s}\;ds\Big|  \le x^\alpha\sum_{P(n)\le y}\frac{1}{n^\alpha}\min\Big(1,\frac{1}{\pi T|\log(x/n)|}\Big)\\
    & \le x^\alpha \sum_{\substack{P(n)\le y\\ |\log(x/n)|>T^{d-1}}}\frac{1}{n^\alpha}\frac{1}{\pi T|\log(x/n)|} + x^\alpha\sum_{\substack{P(n)\le y\\|\log(x/n)|\le  T^{d-1}}}\frac{1}{n^\alpha}\\
    & \le \frac{x^\alpha\zeta(\alpha,y)}{\pi T^d} + e^{\alpha T^{d-1}}\Big[\Psi(xe^{T^{d-1}},y) - \Psi(xe^{-T^{d-1}},y)\Big].
\end{align*}
This completes the proof.
\end{proof}

In using this result we have the problems of performing the integration from $\alpha-iT$ to $\alpha+iT$ and
estimating the number of $y$-smooth integers in the interval $\big(xe^{-T^{d-1}},xe^{T^{1-d}}\big]$.
We turn first to the integral evaluation.
%, and here we discover that the main contribution comes from
%a quite short subinterval of numbers $t$.  

Recall that $B_j=B_j(t) = \sigma_j(x,y) t^j/j!$ and let $B_1^*=B_1^*(t)=t\log x -B_1(t)$. Note that $B_1^*=0$ if $\alpha$ is chosen perfectly.

\begin{lemma}\label{lm:11prep}
For $s=\alpha+it$,  we have
\begin{align*}
&\mathfrak{Re}\Big\{\zeta(s,y)\frac{x^s}{s}\Big\} = \\
&~\frac{x^\alpha\zeta(\alpha,y)}{\alpha^2+t^2}\big(\alpha\cos(B_3+B_1^*+b_5) + t\sin (B_3+B_1^*+b_5)\big)\exp\big\{- B_2+B_4 + a_5\big\},
\end{align*}
where $a_5,b_5$ are real numbers, depending on the choice of $t$, with 
$|a_5+ib_5|\le B_5(t)$.
\end{lemma}
\begin{proof}
 We expand $\phi(\alpha+it,y)=\log\zeta(\alpha+it,y)$ in a Taylor series around $t=0$. There exists some real
 $\xi$ between 0 and $t$ such that
\begin{align*}
\phi(\alpha+it,y) & = \phi(\alpha,y) + it\phi_1 - \frac{t^2}{2}\phi_2 - \frac{it^3}{3!}\phi_3 + \frac{t^4}{4!}\phi_4-i\frac{t^5}{5!}(\alpha+i\xi,y)\\
& = B_0 - iB_1 - B_2 + iB_3 + B_4-i\frac{t^5}{5!}\phi_5(\alpha+i\xi,y).
\end{align*}
Since $\zeta(s,y) = \exp(\phi(s,y))$, we obtain
\begin{align*}
\zeta(s,y)\frac{x^s}{s} & = \frac{\zeta(\alpha,y)x^{\alpha}}{\alpha+it}
\exp\Big\{ it\log x - iB_1 - B_2 + iB_3 + B_4+i\frac{t^5}{5!}\phi_5(\alpha+i\xi,y)\Big\}\\
& = \frac{x^\alpha\zeta(\alpha,y)}{\alpha+it}\exp\Big\{- B_2 +B_4 + i(B_1^*+B_3 )+ i\frac{t^5}{5!}\phi_5(\alpha+i\xi,y)\Big\}.
\end{align*}
Letting $i\phi_5 (\alpha+i\xi)t^5/5! = a_5 + b_5i$, we have
\begin{align*}
&~\zeta(s,y)\frac{x^s}{s}=\\
&  \frac{x^\alpha\zeta(\alpha,y)}{\alpha^2{+}t^2}(\alpha{-}it)\big(\cos(B_1^*+B_3+b_5) + i\sin (B_1^*+B_3+b_5)\big)\exp\big\{- B_2 +B_4+ a_5\big\},
\end{align*}
and taking the real part gives the result.
\end{proof}

%\begin{lemma}\label{lm:11ineq}
 %Let $T_2,T_1,T_0$ be the
%positive solutions to the equations
%\begin{align*}
%B_3(t)-|B_1^*(t)|-B_5(t)&=0,\\
%B_3(t) + |B_1^*(t)|+B_5(t)&=\pi/2,\\
%B_3(t)+|B_1^*(t)|+B_5(t)&=\pi,
%\end{align*}
%respectively.    Assume that $B_3-|B_1^*|-B_5\ge-\pi/2$ on $[0,T_0]$.
%Using the notation of Lemma \ref{lm:11prep},  we have that
%\[
%\min\{\cos(B_3-|B_1|^*-B_5),\cos(B_3+|B_1^*|+B_5)\}\le\cos(B_3+B_1^*+b_5)\le1
%\]
%for $t\in[0,T_2]$ and 
 %\[
 %\cos(B_3+|B_1^*|+B_5)\le\cos(B_3+B_1^*+b_5)\le\cos(B_3-|B_1^*|-B_5)
 %\]
 %for $t\in[T_2,T_0]$.  Further, for $t\in[0,T_1]$ we have
 %\[
 %\sin(B_3-|B_1^*|-B_5)\le\sin(B_3+B_1^*+b_5)\le\sin(B_3+|B_1^*|+B_5),
 %\]
 %and for $t\in[T_1,T_0]$ we have
 %\[
%\min\{\sin(B_3-|B_1^*|-B_5), \sin(B_3+|B_1^*|+B_5)\}\le\sin(B_3+B_1^*+b_5)\le 1.
 %\]
 %\end{lemma}
 %\begin{proof}
 %This uses only the monotonicity of $\sin$ and $\cos$ on the intervals $[-\pi/2,0]$,
 %$[0,\pi/2]$, and $[\pi/2,\pi]$.
 %\end{proof}

%We apply Lemma \ref{lm:11prep} for $t$ close to 0.
%Let $T_0>T_1>0$ be such that $\frac{\partial}{\partial v}f(t,v)<0$ for
%$T_1\le t\le T_0$ and $|v|\le v_0$, where $v_0=|B_1^*(t)|+B_5(t)$.
%We choose $T_1$ as close to 0 as possible and we choose $T_0$ as large
%as possible.  We will have $T_0$ slightly larger than the least positive
%solution to $B_3+v_0=\pi$.

 The main contribution to the integral in Lemma \ref{lm:10} turns out to 
come from the interval $[-T_0,T_0]$, where $T_0$ is fairly small.
 We have
 \[
 \frac1{2\pi i}\int_{\alpha-iT_0}^{\alpha+iT_0}\zeta(s,y)\frac{x^s}s\;ds
= \frac1{2\pi}\int_{-T_0}^{T_0}\zeta(\alpha+it,y)\frac{x^{\alpha+it}}{\alpha+it}\;dt.
\]
Note that the integrand, written as a Taylor series around $s=\alpha$, has real coefficients, so the real part
is an even function of $t$ and the imaginary part is an odd function.  Thus, the integral is
real, and its value is double the value of the integral on $[0,T_0]$.

Consider the cosine, sine combination in Lemma \ref{lm:11prep}:
$$
f(t,v):=\alpha\cos(B_3(t)+v)+t\sin(B_3(t)+v),
$$
and let
$$
v_0(t)=|B_1^*(t)|+B_5(t).
$$
We have, for each value of $t$, the constraint that $|v|\le v_0(t)$.
The partial derivative of $f(t,v)$ with respect to $v$ is zero when
$\arctan(t/\alpha)-B_3(t)\equiv0\pmod{\pi}$.  Let 
$$
u(t)=\arctan(t/\alpha)-B_3(t).
$$
If $u(t)\not\in[-v_0(t),v_0(t)]$, then $f(t,v)$ is monotone in $v$
on that interval; otherwise it has a min or max at $u(t)$.
Let $T_3,T_2,T_1,T_0$ be defined,
respectively, as the least positive solutions of the equations
$$
u(t)=v_0(t),\quad
u(t)=-v_0(t),\quad
u(t)+\pi=v_0(t),\quad
u(t)+\pi=-v_0(t).
$$
Then $0<T_3<T_2<T_1<T_0$.  
We have the following properties for $f(t,v)$:
\begin{enumerate}
\item For $t$ in the interval $[0,T_3]$ we have
$f(t,v)$ increasing for $v\in[-v_0(t),v_0(t)]$, so that
$$f(t,-v_0(t))\le f(t,v)\le f(t,v_0(t)).$$  
\item
For $t$ in the interval $[T_3,T_2]$,
we have $f(t,v)$ increasing for $-v_0(t)\le v\le u(t)$ and then
decreasing for $u(t)\le v\le v_0(t)$.  Thus,
$$\min\{f(t,-v_0(t)) , f(t,v_0(t)) \} \le f(t,v) \le f(t, u(t)).$$
\item
For $t\in[T_2,T_1]$,
$f(t,v)$ is decreasing for $v\in[-v_0(t),v_0(t)]$, so that
$$f(t,v_0(t))\le f(t,v)\le f(t,-v_0(t)).$$
\item
For $t\in[T_1,T_0]$, we have $f(t,v)$ decreasing for $v\in[-v_0(t),u(t)+\pi]$
and increasing for $v\in[u(t)+\pi,v_0(t)]$; that is,
$$f(t,u(t)+\pi) \le f(t,v) \le \max\{f(t,-v_0(t)) , f(t,v_0(t)) \}.$$
\end{enumerate}
Note too that $f(t,v)$
has a sign change from positive to negative in the interval $[T_2,T_1]$.
Let $Z^-,Z^+$ be, respectively, the least positive roots of $f(t,v(t))=0$,
$f(t,-v(t))=0$. 

Let $I_0^+$ be an upper bound for the function appearing in Lemma \ref{lm:11prep} on $[0,T_0]$ using
%the inequalities of Lemma \ref{lm:11ineq} and 
$|a_5|, |b_5|\le B_5$ and the above facts about $f(t,v)$,
and let $I_0^-$ be the corresponding lower bound.  We choose $a_5=B_5$ in
$I_0^+$ when the cos, sin combination is positive, and $a_5=-B_5$ when it
is negative.  For $I_0^-$, we choose $a_5$ in the reverse way.

Let 
\begin{equation}
\label{eq:J0}
J_0^+=\int_0^{T_0}I_0^+(t)\;dt,\quad J_0^-=\int_0^{T_0}I_0^-(t)\;dt.
\end{equation}
We thus have the following result, which is our analogue of Lemma 11 in \cite{HT1}.
\begin{lemma}\label{lm:11}
We have
\[
\frac{x^\alpha\zeta(\alpha,y)}\pi J_0^-\le
  \frac1{2\pi i}\int_{\alpha-iT_0}^{\alpha+iT_0}\zeta(s,y)\frac{x^{s}}{s}\;ds
 \le\frac{x^\alpha\zeta(\alpha,y)}\pi J_0^+.
 \]
 \end{lemma}

 In order to estimate the integral in Lemma \ref{lm:10}
when $|t|>T_0$ we must know something about prime sums to $y$.
\begin{lemma}
\label{lm:8prelim}
We have
\[
\Big| \int_{\alpha+iT_0}^{\alpha+iT} \zeta(s,y)\frac{x^s}{s}\;ds\Big|
\le x^\alpha\zeta(\alpha,y)J_1,
\]
where
\[
J_1:=\int_{T_0}^T\exp\big(-W(y,1,t)\big)\;\frac{dt}{\sqrt{\alpha^2+t^2}}
\]
and
\begin{equation}
\label{def:W}
W(v,w,t) := \sum_{w<p\le v}\frac{1-\cos(t\log p)}{p^\alpha}.
\end{equation}
\end{lemma}
\begin{proof}
For $0\le v\le 1 <t$, equation (3.14) in \cite{HT1} states that
$$(1+4vt/(t-1)^2)^{-1} \le \exp\{-4v/t\}.$$
Applied to (3.17) in \cite{HT1} with $v=(1-\cos(t\log p))/2$, we have that
\begin{align}
\label{eq:ineq317}
\begin{split}
\Big|\frac{\zeta(s,y)}{\zeta(\alpha,y)}\Big| & = \prod_{p\le y}\Big|\frac{1-p^{-\alpha}}{1-p^{-s}}\Big| = \prod_{p\le y}\Big(1+\frac{2(1-\cos(t\log p))}{p^\alpha(1-p^{-\alpha})^2} \Big)^{-1/2}\\
& \le \exp\Big\{-\sum_{p\le y}\frac{1-\cos(t\log p)}{p^\alpha} \Big\}.
\end{split}
\end{align}
This completes the proof.
\end{proof}
 
Our goal now is to find a way to estimate $W(v,w,t)$.          
 The following result is analogous to Lemma 6 in \cite{HT1}.
 \begin{lemma} \label{lm:6}
Let $s$ be a complex number, let $1<w<v$, and define
\begin{align*}
    F_s(v,w):= \sum_{w<p\le v}\frac{\log p}{p^s} - \frac{v^{1-s}-w^{1-s}}{1-s}.
\end{align*}
(i) If $ v\le 10^{19}$ we have
\begin{align*}
|F_s(v,w)| & \le 2(v^{1/2-\alpha} + w^{1/2-\alpha}) + 2|s|\frac{w^{1/2-\alpha}-v^{1/2-\alpha}}{\alpha-1/2}.
\end{align*}
(ii) If $10^{19}\le w\le v$ we have
\begin{align*}
|F_s(v,w)| & \le \varepsilon_w\Big(v^{\beta} + w^{\beta} + |s|\frac{v^{\beta}-w^{\beta}}{\beta}\Big),
\end{align*}
where $\beta=1-\alpha$ and 
\[
\varepsilon_w=\begin{cases}
2.3\cdot  10^{-8},&~~w\in(10^{19},e^{45}],\\
1.2\cdot  10^{-8},&~~w\in(e^{50},e^{55}],\\
1.2\cdot  10^{-9},&~~w\in(e^{50},e^{55}],\\
2.9\cdot  10^{-10},&~~w>e^{55}.
\end{cases}
\]
\end{lemma}
\begin{proof}
(i) By partial summation,
\begin{align*}
\sum_{w<p\le v} \frac{\log p}{p^s} & = \frac{\vartheta(v)}{v^s} - \frac{\vartheta(w)}{w^s} + \int_w^v s\frac{\vartheta(t)}{t^{s+1}}\;dt\\
& = \frac{v^{1-s}-w^{1-s}}{1-s} - \frac{E(v)}{v^s} + \frac{E(w)}{w^s} - \int_w^v s\frac{E(t)}{t^{s+1}}\;dt,
\end{align*}
so by the first part of Proposition \ref{prop:epnt},
\begin{align*}
    |F_s(v,w)| & \le \frac{|E(v)|}{v^\alpha} + \frac{|E(w)|}{w^\alpha} + |s|\int_w^v \frac{E(t)}{t^{1+\alpha}}\;dt\\
    & \le 2v^{1/2-\alpha} + 2w^{1/2-\alpha} + 2|s|\frac{v^{1/2-\alpha}-w^{1/2-\alpha}}{1/2-\alpha}.
\end{align*}

(ii) Similarly, by the second part of Proposition \ref{prop:epnt},
\begin{align*}
    |F_s(v,w)| & \le \frac{|E(v)|}{v^\alpha} + \frac{|E(w)|}{w^\alpha} + |s|\int_w^v \frac{E(t)}{t^{1+\alpha}}\;dt \le \varepsilon_w\Big(v^{1-\alpha} + w^{1-\alpha} + |s|\int_w^v \frac{dt}{t^{\alpha}}\Big)\\
    & = \varepsilon_w\Big(v^{1-\alpha} + w^{1-\alpha} + |s|\frac{v^{1-\alpha}-w^{1-\alpha}}{1-\alpha}\Big).
\end{align*}
\end{proof}

The following result plays the role of Corollary 6.1 in \cite{HT1}.
\begin{lemma}
\label{lm:6.1}
For $t\in\R$, $z>1$, and $\beta=1-\alpha$, let
\begin{align*}
    \delta_z := t\log z- \arctan(t/\beta).
\end{align*}
(i) For $1427\le w< v\le 10^{19}$ we have that $W(v,w,t)\ge W_0(v,w,t)$, where
\begin{align*}
W_0(v,w,t)\log v & =\frac{v^{\beta}-w^\beta}{\beta} - \frac{v^{\beta}\cos\delta_v-w^\beta\cos\delta_w}{\sqrt{\beta^2+t^2}}\\
& \qquad - 4(v^{1/2-\alpha}+w^{1/2-\alpha}) - 2(\alpha+|s|)\frac{w^{1/2-\alpha}-v^{1/2-\alpha}}{\alpha-1/2}.
\end{align*}
(ii) For $10^{19}\le w< v$ we have that $W(v,w,t)\ge W_0(v,w,t)$, where
\begin{align*}
W_0(v,w,t)\log v & = \frac{v^{\beta}-w^\beta}{\beta} - \frac{v^{\beta}\cos\delta_v-w^\beta\cos\delta_w}{\sqrt{\beta^2+t^2}}\\
& \qquad -2\varepsilon_w\big(v^\beta + w^\beta\big) -  \varepsilon_w(\alpha+|s|)\Big(\frac{v^\beta-w^\beta}{\beta}\Big).
\end{align*}
\end{lemma}
\begin{proof}
We apply Lemma \ref{lm:6} with $s=1-\beta$ and $s=1-\beta+it$, and take the real part of the difference. Letting the difference of the sums be $S$, we have that
\begin{align*}
S : & = \sum_{w<p\le v}\Big( \frac{\log p}{p^{1-\beta}} - \frac{\log p}{p^{1-\beta+it}} \Big)= \sum_{w<p\le v} \frac{\log p}{p^{1-\beta}}(1-p^{-it}), \text{ so}\\
\mathfrak{Re}(S) & = \sum_{w<p\le v} \frac{\log p}{p^{1-\beta}}(1-\cos(t\log p)),
\end{align*}
which is the sum we wish to bound.

For a positive real number $z$, let $S_z :  = \frac{z^{\beta}}{\beta} - \frac{z^{\beta-it}}{\beta-it} $.
We have that
\begin{align*}
S_z& =  \frac{z^{\beta}}{\beta}\Big(1-\frac{\beta}{\beta-it}z^{-it}\Big) = \frac{z^{\beta}}{\beta}\Big(1-\beta\frac{\beta+it}{\beta^2+t^2}e^{-it\log z}\Big)\\
& = \frac{z^{\beta}}{\beta}\Big(1-\beta\frac{\beta+it}{\beta^2+t^2}[\cos(t\log z)-i\sin(t\log z)]\Big),
\end{align*}
so by Lemma \ref{lm:6.1},
\begin{align*}
\mathfrak{Re}(S_z) & = \frac{z^{\beta}}{\beta}\Big(1-\frac{\beta}{\beta^2+t^2}[\beta\cos(t\log z) + t\sin(t\log z)]\Big)\\
& = \frac{z^{\beta}}{\beta}\Big(1-\frac{\beta}{\sqrt{\beta^2+t^2}}\Big[\frac{\beta\cos(t\log z)}{\sqrt{\beta^2+t^2}} + \frac{t\sin(t\log z)}{\sqrt{\beta^2+t^2}}\Big]\Big)\\
& = \frac{z^{\beta}}{\beta}\Big(1-\frac{\beta}{\sqrt{\beta^2+t^2}} \cos(t\log z+\arctan(\beta/t))\Big)
 = \frac{z^{\beta}}{\beta}\Big(1-\frac{\beta\cos\delta_z}{\sqrt{\beta^2+t^2}}\Big).
\end{align*}
Thus,
\begin{align} \label{eq:6.1.}
\mathfrak{Re}(S_v-S_w) = \frac{v^{\beta}-w^\beta}{\beta}-\frac{v^{\beta}\cos\delta_v-w^\beta\cos\delta_w}{\sqrt{\beta^2+t^2}}.
\end{align}

Recalling the definition of $F_s(v,w)$, we have
%\begin{align*}
%\sum_{w<p\le v} \frac{\log p}{p^{1-\beta}} - \frac{v^{\beta}-w^{\beta}}{\beta} & = F_{1-\beta}(v,w), \text{ and }\\
%\sum_{w<p\le v} \frac{\log p}{p^{1-\beta+it}} -  \frac{v^{\beta-it}-w^{\beta-it}}{\beta-it} & = F_{1-\beta+it}(v,w).
%\end{align*}
%Subtracting the second from the first and taking the real part, we have
\begin{align*}
\mathfrak{Re}(S)  & = \mathfrak{Re}(S_v-S_w + F_\alpha(v,w) -F_{s}(v,w))\\
& \ge \mathfrak{Re}(S_v-S_w) - |F_{\alpha}(v,w)| - |F_{s}(v,w)|
\end{align*}
which gives the desired result by \eqref{eq:6.1.} and Lemma \ref{lm:6}.
\end{proof}

 From Lemma \ref{lm:8prelim}, we see that a goal is 
to bound $W(y,1,t)$ from below, and pieces of this sum are bounded by Lemma \ref{lm:6.1}.
 Ideally, if $y$ were sufficiently small $W$ could be 
computed directly and the problem settled. In practice $W$ might only be computed up to some
convenient number $L$, suitable for numerical integration, after which the analytic bound $W_0(y,w,t)$ may be used. Still, there are further refinements to be made. Just as $x/\log x$ loses out to $\li(x)$, $W_0$ on a long interval is smaller than $W_0$ summed on a partition of the interval into
shorter parts.  This plan is reflected in the following lemma.

\begin{lemma}
\label{lm:8}
If $v,w$ satisfy the hypotheses of Lemma \ref{lm:6}, let
\[
W_*(v,w,t):=W_0(v/e^{\lfloor\log(y/w)\rfloor},w,t)+\sum_{j=0}^{\lfloor\log(v/w)\rfloor-1}W_0(v/e^j,v/e^{j+1},t).
\]
Suppose that $w,L$ satisfy $1427,L\le w$.
If $y\le 10^{19}$, then
\[
J_1\le\int_{T_0}^T\exp\big(-W_*(y,w,t)-W(L,1,t)\big)\;\frac{dt}{\sqrt{\alpha^2+t^2}}.
\]
If $y>e^{55}$ and $1427,L\le w\le 10^{19}$,  let
\begin{align*}
W_1&=W_*(10^{19},w,t),~~W_2=W_*(e^{45},10^{19},t),~~W_3=W_*(e^{50},e^{45},t),\\
W_4&=W_*(e^{55},e^{50},t),~~W_5=W_*(y,e^{55},t).
\end{align*}
Then
\[
J_1\le\int_{T_0}^T\exp\big(-W_1-W_2-W_3-W_4-W_5-W(L,1,t)\big)\;\frac{dt}{\sqrt{\alpha^2+t^2}}.
\]
\end{lemma}
We remark that if $10^{19}<y\le e^{55}$, then there is an appropriate inequality for $J_1$
involving fewer $W_j$'s.  If $y$ is much larger than our largest example of $y=10^{35}$,
one might wish to use better approximations to $\vartheta(y)$ than
were used in Proposition \ref{prop:epnt}.
\begin{proof}
  If $1427\le w< v$ and $[w,v]$ satisfy the hypotheses of Lemma \ref{lm:6}, we have
\begin{align*}
W(v,w,t)&=W(v/e^{\lfloor\log(v/w)\rfloor},w,t)+\sum_{j=0}^{\lfloor\log(v/w)\rfloor-1}W(v/e^j,v/e^{j+1},t)\\
&\ge W_0(v/e^{\lfloor\log(v/w)\rfloor},w,t)+\sum_{j=0}^{\lfloor\log(v/w)\rfloor-1}W_0(v/e^j,v/e^{j+1},t).
\end{align*}
The result then follows from Lemma \ref{lm:8prelim}.
\end{proof}
\begin{remark}
\label{rmk:lm8}
We implement Lemma \ref{lm:8} by choosing $L$ as large as possible so as not to
interfere overly with numerical
integration.  We have found that $L=10^6$ works well. 
The ratio $e$ in the definition of $W_*$ is convenient, but might be tweaked for
 slightly better results.  The individual terms in the sum $W(L,1,t)$
are as in \eqref{def:W}, except for the first 30 primes, where instead
we forgo using the inequality in 
\eqref{eq:ineq317}, using instead the slightly larger expression
\[
\frac12\log\Big(1+\frac{2(1-\cos(t\log p))}{p^\alpha(1-p^{-\alpha})^2}\Big).
\]
 
 We choose $w$ as a function
$w(t)$ in such a way that the bound in Lemma \ref{lm:6.1} is minimized.
For simplicity, we ignore the oscillating terms, i.e., we set
\begin{align*}
  \frac{\partial}{\partial w}&\Big[ -w^\beta/\beta- 4w^{1/2-\alpha}+ 2(\alpha+|s|)w^{1/2-\alpha}/(1/2-\alpha)\Big]\\
 &~ =  -w^{\beta-1} - 4w^{-1/2-\alpha} /(1/2-\alpha)+ 2(\alpha + |s|)w^{-1/2-\alpha}
\end{align*}
equal to 0.
Multiplying by $w^{1/2+\alpha}$ and solving for $w$ gives
$$ 
w(\alpha,t): = \Big(\frac{4}{\alpha-1/2} + 2\alpha + 2\sqrt{\alpha^2+t^2} \Big)^2.
$$
We let
$$
w(t):=\max\{L,w(\alpha,t)\}.
$$
\end{remark}

Our next result, based on  \cite[Lemma 9]{HT1}, gives a bound on the
number of $y$-smooth integers in a short interval.
\begin{lemma} \label{lm:9}
Let $0<d<1$, $T>1$ be such that
$z := (e^{2T^{d-1}}-1)^{-1}>1$.
We have
$$
    \Psi(xe^{T^{1-d}},y)  - \Psi(xe^{-T^{1-d}},y) \le
     e^{\alpha^2/2z^2-\alpha T^{d-1}}x^\alpha\zeta(\alpha,y)\sqrt{\frac{2e}{\pi}}\frac{J_2}{z}.
$$
where, with $W(y,w,t)$ as in Lemma \ref{lm:6.1},
$$
J_2:=\int_0^{\infty}\exp\Big\{-\frac{t^2}{2z^2}-W(y,1,t)\Big\}\;dt.
$$
\end{lemma}
\begin{proof}
Let $\xi = xe^{-T^{d-1}}$,  so that
\begin{equation}
\label{eq:relabel}
\Psi(xe^{T^{d-1}},y) -\Psi(xe^{-T^{d-1}},y) 
    = \Psi(\xi+\xi/z,y)-\Psi(\xi,y).
\end{equation}
For $\xi< n\le \xi+\frac{\xi}{z}$, we have that
$$1>\frac{\xi}{n}\ge \Big(1+\frac{1}{z}\Big)^{-1},$$
so
$0>\log(\xi/n) \ge -\log(1+1/z) \ge -\frac{1}{z}$, which implies that
$0<[z\log(\xi/n)]^2 \le 1.$
Thus,
$$
\Psi(\xi+\xi/z,y)-\Psi(\xi,y) = \sum_{\substack{P(n)\le y\\\xi< n\le \xi+\xi/z}} 1 \le \sqrt{e}\sum_{\substack{P(n)\le y\\\xi< n\le \xi+\xi/z}}\exp\Big\{-\frac{1}{2}[z\log(\xi/n)]^2\Big\}.
$$
For $\sigma,v\in\R$, we have the formula
$$
e^{-v^2/2} = \frac{1}{\sqrt{2\pi}}e^{\sigma^2/2-\sigma v}\int_{-\infty}^{+\infty}\exp\Big\{-\frac{1}{2}t^2 +it(\sigma-v)\Big\}\;dt
$$
Letting $\sigma=\alpha/z$, $v=-z\log(\xi/n)$, we obtain
\begin{align*}
\Psi(\xi+\xi/z,y)-&\Psi(\xi,y) \\
 \le &\sqrt{\frac{e}{2\pi}}\sum_{\substack{P(n)\le y\\\xi< n\le \xi+\xi/z}}e^{\sigma^2/2-\sigma v}\int_{-\infty}^{+\infty}\exp\Big\{-\frac{1}{2}t^2 +it(\sigma-v)\Big\}\;dt\\
=&  e^{\alpha^2/2z^2}\sqrt{\frac{e}{2\pi}}\int_{-\infty}^{+\infty}\exp\Big\{-\frac{1}{2}t^2 +it\alpha/z\Big\}\sum_{\substack{P(n)\le y\\\xi< n\le \xi+\xi/z}}\Big(\frac{\xi}{n}\Big)^{\alpha+it z}\;dt.
\end{align*}
Since $\alpha \le 1 \le z$, changing variables $t\mapsto t/z$ and taking the modulus gives
\begin{align*}
\Psi(\xi+\xi/z,y)&-\Psi(\xi,y) \\
 \le &z^{-1}e^{\alpha^2/2z^2}\sqrt{\frac{e}{2\pi}}\int_{-\infty}^{+\infty}\exp\Big\{-\frac{t^2}{2z^2} +it\alpha/z^2\Big\}\sum_{\substack{P(n)\le y\\\xi< n\le \xi+\xi/z}}\Big(\frac{\xi}{n}\Big)^{\alpha+it}\;dt\\
 \le& \frac{\xi^\alpha}{z}e^{\alpha^2/2z^2}\sqrt{\frac{e}{2\pi}}\int_{-\infty}^{+\infty}e^{-t^2/2z^2}|\zeta(\alpha+it,y)|\;dt\\
=&\frac{\xi^\alpha}{z}e^{\alpha^2/2z^2}\sqrt{\frac{2e}{\pi}}\int_0^\infty
e^{-t^2/2z^2}|\zeta(\alpha+it,y)|\;dt.
\end{align*}
This last integral may be estimated by the method of Lemma \ref{lm:8prelim},
giving
\[
\int_0^\infty e^{-t^2/2z^2}|\zeta(\alpha+it,y)|\;dt
\le \zeta(\alpha,y)\int_0^\infty
\exp\Big(-\frac{t^2}{2z^2}-W(y,1,t)\Big)\;dt
= \zeta(\alpha,y)J_2.
\]
We have
 \[
    \Psi(\xi+\xi/z,y)-\Psi(\xi,y)  \le \xi^\alpha\zeta(\alpha,y)e^{\alpha^2/2z^2}\sqrt{\frac{2e}{\pi}}\frac{J_2}{z},
 \]
 and the lemma now follows from \eqref{eq:relabel} and the definition of $\xi$.
\end{proof}
\begin{remark}
For $t$ large, say $t>2z\log z$, we can ignore the term $W(y,1,t)$ in $J_2$,
getting a suitably tiny numerical estimate for the tail of this 
rapidly converging integral.  The part for $t$ small
may be integrated numerically with $w(t), L$ as in Remark \ref{rmk:lm8}. 
\end{remark}

With these lemmas, we now have our principal result.
\begin{theorem}\label{thm:main}
Let  $d,T,z$ be as in Lemma \ref{lm:9}, let $J_0^\pm$ be as in
\eqref{eq:J0}, $J_1$ as in Lemma \ref{lm:8prelim}, and
$J_2$ as in Lemma \ref{lm:9}.
We have
$$
\Psi(x,y) \ge \frac{x^\alpha\zeta(\alpha,y)}{\pi}\Big( J_{0}^- - J_1 - T^{-d} -e^{\alpha^2/2z^2} \sqrt{2\pi e}\frac{J_2}{z}\Big)
$$
and
$$
\Psi(x,y) \le \frac{x^\alpha\zeta(\alpha,y)}{\pi}\Big(J_{0}^+ + J_1 + T^{-d} + e^{\alpha^2/2z^2}\sqrt{2\pi e}\frac{J_2}{z}\Big).
$$
\end{theorem}

\section{Computations}

In this section we give some guidance on how, for a given pair $x,y$, the numbers $\alpha$,
$\zeta(\alpha,y)$, and $\sigma_j$ for $j\le 5$ may be numerically approximated.
Further, we discuss how these data may be used to numerically approximate
$\Psi(x,y)$ via Theorem \ref{thm:main}.

\subsection{Computing $\alpha$}
Given a number $a\in(0,1)$ and a large number $y$ we may obtain upper and lower bounds
for the sum
$$
\sigma_1(a,y)=\sum_{p\le y}\frac{\log p}{p^a-1}.
$$
First, we choose a moderate bound $w_0\le y$ where we can compute the sum $\sigma_1(a,w_0)$
relatively easily, such as $w_0=179{,}424{,}673$, the ten-millionth prime.  The sum
\begin{equation}
\label{eq:basicsum}
\sum_{w_0<p\le y}\frac{\log p}{p^a}
\end{equation}
may be approximated easily with Proposition \ref{prop:epnt} and partial summation.
Let $l^-(a,w_0,y)$ be a lower bound for this sum and let $l^+(a,w_0,y)$ be an upper bound.
Then
$$
l^-(a,w_0,y)+\sigma_1(a,w_0) \le \sigma_1(a,y)\le \frac{w_0^a}{w_0^a-1}l^+(a,w_0,y)+\sigma_1(a,w_0).
$$
We choose $\alpha$ as a number $a$ where $\log x$ lies between these two bounds.
If a given trial for $a$ is too small, this is detected by our lower bound for $\sigma_1(a,y)$ lying
above $\log x$, and if $a$ is too large, we see this if our upper bound for $\sigma_1(a,y)$
lies below $\log x$.  It does not take long via linear interpolation to find a reasonable choice for $\alpha$.
While narrowing in, one might use a less ambitious choice for $w_0$.

The partial summation used to estimate \eqref{eq:basicsum} and similar sums
may be summarized in the following result.
\begin{lemma}
\label{lem:ps}
Suppose $f(t)$ is positive and $f'(t)$ is negative on $[w_0,w_1]$.  Suppose too
that $t-2\sqrt{t}<\vartheta(t)\le t$ on $[w_0,w_1]$.  Then
\begin{align*}
\int_{w_0}^{w_1}(1-1/\sqrt{t})f(t)\;dt&+(w_0-\vartheta(w_0)-2\sqrt{w_0})f(w_0)\\
&\le\sum_{w_0<p\le w_1}f(p)\log p
\le\int_{w_0}^{w_1}f(t)\;dt+(w_0-\vartheta(w_0))f(w_0).
\end{align*}
\end{lemma}
Because of Proposition \ref{prop:epnt}, the condition on $\vartheta$ holds
if $[w_0,w_1]\subset[1427,10^{19}]$.  For intervals beyond $10^{19}$, it
is easy to fashion an analogue of Lemma \ref{lem:ps} using the other
estimates of Proposition \ref{prop:epnt}.

\subsection{Computing $\sigma_0=\log\zeta(\alpha,y)$ and the other $\sigma_j$'s}

Once a choice for $\alpha$ is computed it is straightforward to compute $\sigma_0$ and
the other $\sigma_j$'s.

We have
$$
\sigma_0(\alpha,y)=\sum_{p\le y}-\log(1-p^{-\alpha}).
$$
We may compute this sum up to some moderate $w_0$ as with the $\alpha$ computation.
For the range $w_0<p\le y$ we may approximate the summand by $p^{-\alpha}$ and
sum this over $(w_0,y]$ using partial summation (Lemma \ref{lem:ps})
and Proposition \ref{prop:epnt}, say 
a lower bound is $l_0^-$ and an upper bound is $l_0^+$.  Then
$$
l_0^-+\sigma_0(\alpha,w_0)\le \sigma_0(\alpha,y)\le\frac{-\log(1-w_0^{-\alpha})}{w_0^{-\alpha}}l_0^++\sigma_0(\alpha,w_0).
$$
The other $\sigma_j$'s are computed in a similar manner.  

\subsection{Data}

We record our calculations of $\alpha$ and the numbers $\sigma_j$ for two examples.  Note that
we obtain bounds for $\zeta$ via $\sigma_0=\log\zeta$.
\begin{figure}[H]
  \caption{Data.} \label{fig:data}
%  \[\begin{array}{c|cccc}
  \[\begin{array}{c|cc}
% x & 10^{100} & 10^{100} & 10^{100} & 10^{500} \\
% y & 10^{14} & 10^{15} & 10^{16} & 10^{35} \\
 x & 10^{100} & 10^{500} \\
 y & 10^{15} & 10^{35} \\
\hline
% \alpha & .901862 & .9111582 & .919143  & .94932677 \\
\alpha & .9111581 & .94932677 \\
\hline
%&\zeta^- & 631076 & 352095 & 211278 & 20918054420\\
%&\zeta^+ & 631457 & 352300 & 211398 & 20924846831\\
\zeta & 352{,}189\pm16 & 2.09222\cdot 10^{10}\pm5\cdot 10^5 \\
\hline
%\sigma_1^-&  & 230.2577 &  & 1151.2909 \\
%\sigma_1^+& & 230.2588 &  &  1151.2935\\
\sigma_1^{*} & 4.3\cdot 10^{-4} & 5.6\cdot 10^{-4}\\
\hline
%\sigma_2^- & 5417.6 & 5763.3 & 6105.0 & 71688.5\\
%\sigma_2^+ & 5418.0 & 5763.7 & 6105.5 & 71689.6\\
\sigma_2 & 5{,}763.47\pm0.03 &  71{,}689.2 \pm 0.02\\
\hline
%\sigma_3^- & 116246 & 134146 & 153125 & 4399099\\
%\sigma_3^+ & 116254 & 134156 & 153135 & 4399161\\
\sigma_3 & 159{,}066.8\pm0.5& 4{,}779{,}948.5\pm0.5\\
\hline
%\sigma_4^- & 3828788 & 4634309 & 5542947 & 330272458\\
%\sigma_4^+ & 3829024 & 4634603 & 5543315 & 330277628
\sigma_4 & 4{,}604{,}079\pm8& 330{,}260{,}722\pm21\\
\hline
\sigma_5^+ & 1.3725\cdot 10^8 &  2.3353\cdot 10^{10}
\end{array}\]
\end{figure}

Note that $\sigma_1^{*}$ is an upper bound for $|\sigma_1-\log x|$,
and $\sigma_5^+$ is an upper bound for $\sigma_5$.

The functions $\alpha(x,y)$ and $\sigma_j(x,y)$ are of interest in their own 
right. A simple observation from their definitions allows for more general bounds on $\alpha$ and $\sigma_j$ using the data in Figure \ref{fig:data}, as described in the following remark.

\begin{remark}
For pairs $x,y$ and $x',y'$, if $x\ge x'$ and $y \le y'$ then
$\alpha(x,y)\le \alpha(x',y')$. Similarly, if $\alpha(x,y)\ge \alpha(x',y')$ and $y \le y'$ then $\sigma_j(x,y)\le \sigma_j(x',y')$.
\end{remark}

\subsection{A word on numerical integration}

The numerical integration needed to estimate $J_1,J_2$ is difficult, especially when
we choose a large value of $L$, like $L=10^6$.  We performed these integrals independently
on both Mathematica and Sage platforms.  It helps to segment the range of integration,
but even so, the software can report an error bound in addition to the main estimate.
In such cases we have always added on this error bound and then rounded up, since we
seek upper bounds for these integrals.   In a case where one wants to be assured
of a rigorous estimate, there are several options, each carrying some costs.  One can
use a Simpson or midpoint quadrature with a mesh say of $0.1$ together with a careful
estimation of the higher derivatives needed to estimate the error.  An alternative is
to do a Riemann sum with mesh $0.1$, where on each interval and for each separate
cosine term appearing, the maximum contribution is calculated.  If this is done with
$T=4\cdot10^5$ and $L=10^6$, there would be magnitude $10^{11}$ of these
calculations.  The extreme value of the cosine contribution would either be at an
endpoint of an interval or $-1$ if the argument straddles a number that is $\pi\bmod 2\pi$.
We have done a mild form of this method in our estimation of the integrals $J_0^\pm$.

\subsection{Example estimates}

We list some example values of $x,y$ and the corresponding estimates in the figure below.
\begin{figure}[H]
  \caption{Results.} \label{fig:res}
  \[\begin{array}{c|cc}
x &  10^{100} &  10^{500} \\
y &  10^{15} &  10^{35} \\
\hline
T_3 & .00642708 & .00114940 \\
T_2 & .00644109 & .00115038 \\
Z^- & .0385260 & .0124202 \\
Z^+ & .0403125 & .0127461 \\
T_1 & .0478624 & .0155272 \\
T_0 & .0514483 & .0161799 \\
T &  4\cdot 10^5 & 10^9 \\
d & 0.57 & 0.58 \\
\hline
J_{0}^- &  1.78554\cdot10^{-2} &  4.90043\cdot10^{-3} \\
J_{0}^+ &  1.80312\cdot10^{-2} &  4.92738\cdot10^{-3} \\
J_1 &  7.236\cdot10^{-4} &  1.717\cdot10^{-6} \\
J_2 &  1.758\cdot10^{-2} &  4.745\cdot10^{-3} \\
\hline\\
\Psi^- &  2.3302\cdot10^{94}& 1.4989\cdot10^{482} \\
\Psi^+ &  2.9227\cdot10^{94} &  1.5118\cdot10^{482} \\
\end{array}\]
\end{figure}

\section{Appendix}

We prove the following theorem.
\begin{theorem}[Granville and Soundararajan]
\label{thm:GS}
If $3\le y\le x$ and $1/\log y\le\sigma\le1$, then
$$
\Psi(x,y)\le1.39\frac{y^{1-\sigma}}{\log x}x^\sigma\zeta(\sigma,y).
$$
\end{theorem}
\begin{proof}
By the identity $\log n=\sum_{d|n}\Lambda(d)$, we have
\begin{align*}
\sum_{\substack{n\le x\\P(n)\le y}}\log n&=\sum_{\substack{m\le x\\P(m)\le y}}
\sum_{\substack{d\le x/m\\P(d)\le y}}\Lambda(d)
=\sum_{\substack{m\le x\\P(m)\le y}}\sum_{\substack{p\le\min\{y,x/m\}}}\log p\Big\lfloor\frac{\log(x/m)}{\log p}\Big\rfloor\\
&\le\sum_{\substack{m\le x\\P(m)\le y}}\pi\big(\min\{y,x/m\}\big)\log(x/m).
\end{align*}
Thus,
\[
\Psi(x,y)\log x=\sum_{\substack{n\le x\\P(n)\le y}}(\log n+\log(x/n))
\le\sum_{\substack{n\le x\\P(n)\le y}}\big(1+\pi(\min\{y,x/n\})\big)\log(x/n).
\]
Using the estimates in \cite{RS1} we see that the maximum of
$(1+\pi(t))/(t/\log t)$ occurs at $t=7$, so that
\[
1+\pi(t)<1.39t/\log t
\]
for all $t>1$.  The above estimate then gives
\[
\Psi(x,y)\log x<1.39\sum_{\substack{x/y<n\le x\\P(n)\le y}}x/n
+1.39\sum_{\substack{n\le x/y\\P(n)\le y}}y\log(x/n)/\log y.
\]
We now note that if $1/\log y\le\sigma\le1$, then
\[
y^{1-\sigma}(x/n)^\sigma\ge\begin{cases}
x/n,&\hbox{ if }x/y<n\le x,\\
y\log(x/n)/\log y,&\hbox{ if }n\le x/y.
\end{cases}
\]
Indeed, in the first case, 
since $t^{1-\sigma}$ is non-decreasing
in $t$, we have $(x/n)^{1-\sigma}\le y^{1-\sigma}$.  And in
the second case, 
since $t^{-\sigma}\log t$ is decreasing in $t$ for $t\ge y$,
we have $(x/n)^{-\sigma}\log(x/n)\le y^{-\sigma}\log y$.

We thus have
\[
\Psi(x,y)\log x<1.39\sum_{\substack{n\le x\\P(n)\le y}}y^{1-\sigma}(x/n)^\sigma
<1.39y^{1-\sigma}x^\sigma\zeta(\sigma,y).
\]
This completes the proof.
\end{proof}

\section*{Acknowledgments}
We warmly thank Jan B\"uthe, Anne Gelb, Habiba Kadiri, Dave Platt, 
Brad Rodgers, Jon Sorenson, Tim Trudgian, and
John Voight for their interest and help.  We are also very appreciative of
Andrew Granville and Kannan Soundararajan for allowing us to include
their elementary upper bound prior to the publication of their book.
The first author was partially supported by a Byrne Scholarship
at Dartmouth.  The second author was partially supported by NSF grant number DMS-1440140
while in residence at the Mathematical Sciences Research Institute in Berkeley.

\bibliographystyle{amsplain}

\end{document}